\documentclass[amscd,amssymb]{amsart}

\usepackage{amssymb}
\usepackage[all]{xy}

\newtheorem{thm}{Theorem}[section]
\newtheorem{lem}[thm]{Lemma}
\newtheorem{cor}[thm]{Corollary}
\newtheorem{prop}[thm]{Proposition}
\newtheorem{conj}[thm]{Conjecture}

\theoremstyle{definition}

\newtheorem{defn}[thm]{Definition}
\newtheorem{defns}[thm]{Definitions}

\newtheorem{prob}[thm]{Problem}
\newtheorem{notation}[thm]{Notation}
\newtheorem{ex}[thm]{Example}

\theoremstyle{remark}

\newtheorem{rem}[thm]{Remark}

\numberwithin{equation}{section}

\newcommand{\thmref}[1]{Theorem~\ref{#1}}
\newcommand{\corref}[1]{Corollary~\ref{#1}}
\newcommand{\secref}[1]{\S\ref{#1}}
\newcommand{\conjref}[1]{Conjecture~\ref{#1}}
\newcommand{\probref}[1]{Problem~\ref{#1}}
\newcommand{\propref}[1]{Proposition~\ref{#1}}
\newcommand{\lemref}[1]{Lemma~\ref{#1}}

\newcommand{\exref}[1]{Example~\ref{#1}}
\newcommand{\remref}[1]{Remark~\ref{#1}}

\newcommand{\Ext}{\operatorname{Ext}}

\newcommand{\Ind}{\operatorname{Ind}}

\newcommand{\im}{\operatorname{im}}

\newcommand{\Rad}{\operatorname{Rad}}

\newcommand{\A}{{\mathcal  A}}

\newcommand{\U}{{\mathcal  U}}
\newcommand{\K}{{\mathcal  K}}
\newcommand{\HH}{{\mathcal  H}}
\newcommand{\FF}{{\mathcal  F}}

\newcommand{\Z}{{\mathbb  Z}}

\newcommand{\F}{{\mathbb  F}}

\newcommand{\sm}{\wedge}

\newcommand{\ra}{\rightarrow}
\newcommand{\xra}{\xrightarrow}

\newcommand{\hra}{\hookrightarrow}
\newcommand{\era}{\twoheadrightarrow}

\begin{document}

\title[group cohomology nilpotence]{Nilpotence in group cohomology}

 \author[Kuhn]{Nicholas J.~Kuhn}
 \address{Department of Mathematics \\ University of Virginia \\ Charlottesville, VA 22904}
 \email{njk4x@virginia.edu}
\thanks{This research was partially supported by N.S.F. grants 0604206 and 0967649.}

 \date{February 21, 2010.}


 \begin{abstract}

 We study bounds on nilpotence in $H^*(BG)$, the mod $p$ cohomology of the classifying space of a compact Lie group $G$. Part of this is a report of our previous work on this problem, updated to reflect the consequences of Peter Symonds recent verification of Dave Benson's Regularity Conjecture.  New results are given for finite $p$--groups, leading to good bounds on nilpotence in $H^*(BP)$ determined by the subgroup structure of the $p$--group $P$.

\end{abstract}

\maketitle

\section{Introduction}

Fixing a prime $p$, let $H^*(BG)$ denote the mod $p$ cohomology ring of the classifying space of a compact Lie group $G$.  This is a graded commutative $\F_p$--algebra of great interest as it is the home for mod $p$ characteristic classes of principal $G$ bundles.  Furthermore, when $G$ is finite, this ring identifies with $\Ext^*_{\F_p[G]}(\F_p, \F_p)$, and so contains much detailed module theoretic information.

Precise calculation of $H^*(BG)$ can be daunting, particularly when $G$ is a finite $p$--group.  In this paper we study nilpotence in $H^*(BG)$.  We offer some updates of our previous work in \cite{k5}, together with new results in the finite $p$--group case.

We should be more precise about what we mean by `nilpotence'.

Let $\Rad(G)$ be the nilradical of the graded $\F_p$--algebra $H^*(BG)$. One can define an `algebraic' nilpotence degree as follows.

\begin{defn}  Define $d^{alg}(G)$ to be the maximal $d$ such that $\Rad(G)^{d} \neq 0$.
\end{defn}

As the mod $p$ cohomology of a topological space, $H^*(BG)$ is in the category $\U$, the category of modules over the mod $p$ Steenrod algebra $\A_p$ which satisfy the unstable condition. Following Hans-Werner Henn, Jean Lannes, and Lionel Schwartz in \cite{hls1}, one can define a `topological' nilpotence degree as follows.  Let $\Sigma^d M$ denote the $d$th suspension (upward shift) of a graded module $M$.

\begin{defn}  Define $d^{\U}(G)$ to be the maximal $d$ such that $H^*(BG)$ contains a nonzero submodule of the form $\Sigma^d M$, with $M \in \U$.
\end{defn}

This definition is clearly just dependent on the $\A_p$ module structure of $H^*(BG)$, but results in \cite{hls1} allows for comparison with $d^{alg}(G)$.  As will be reviewed in \secref{old results revisited},
\begin{equation*}
d^{alg}(G) \leq
\begin{cases}
d^{\U}(G) & \text{if } p=2 \\ d^{\U}(G) + r(G) & \text{if } p \text{ is odd}.
\end{cases}
\end{equation*}
Here $r(G)$ is the maximal rank of an elementary abelian $p$--subgroup of $G$.

Our goal here is to describe how to calculate $d^{\U}(G)$, and, in particular, to give good group theoretic upper bounds.  We note that $d^{\U}(\Z/p) = 0$ and $d^{\U}(G \times H) = d^{\U}(G) + d^{\U}(H)$.  Furthermore, by transfer arguments, $d^{\U}(G) \leq d^{\U}(P)$, if $P$ is a $p$--Sylow subgroup of a finite group $G$, and a similar inequality holds for a general compact Lie group $G$, with $P$ now the evident extension of a maximal torus $T$ by a $p$--Sylow subgroup of $N_G(T)/T$.

\subsection{A general bound on $d^{\U}(G)$.}

\begin{notation} Throughout the paper, we let $E$ denote an elementary abelian $p$--group, i.e.~ a group isomorphic to $(\Z/p)^r$ for some $r$. We let $E^{\#}$ denote the dual of $E$. As mentioned above, $r(G)$ will denote the maximal rank of $E < G$.  Let $C(G) < G$ be the maximal central elementary abelian $p$--subgroup, and let $c(G)$ denote its rank.
\end{notation}

We recall from \cite{k5} the definition of a key invariant.

\begin{defn} Via restriction, $H^*(BC(G))$ is a finitely generated $H^*(BG)$--module, and we let $e(G)$ denote the top degree of a generator.
\end{defn}

\begin{thm} \label{d(G) thm}  If $G$ is compact Lie, then
$$ \max_{\substack{E < G \\ r(E) = r(G)}} \{ e(C_G(E))- \dim(C_G(E)) \} \leq d^{\U}(G) \leq \max_{E<G} \{ e(C_G(E))- \dim(C_G(E))\}.$$
\end{thm}
Here $\dim(G)$ denotes the dimension of a Lie group $G$ as a manifold, and so is 0 if $G$ is finite.  

In the theorem, the indexing for the upper bound can be restricted to $E$ which contain $C(G)$.  Thus the lower bound equals the upper bound when $c(G) = r(G)$, i.e.~ $G$ is {\em $p$--central} -- a group in which every element of order $p$ is central -- and, in that case, $d_0^{\U}(G) =  e(G) - \dim G$.

The proof of \thmref{d(G) thm} is given in \secref{old results revisited}. Most of this is a review and slight reorganization of work in \cite{k5}, with results extended to all compact Lie groups.  Some of our results were previously conditional on the verification of Dave Benson's Regularity Conjecture \cite{b} which conjectured the vanishing of certain local cohomology groups.  Happily, this is now a theorem of Peter Symonds \cite{symonds}, and we make very precise how the vanishing of local cohomology groups allows for improvement on \thmref{d(G) thm}.

\subsection{Bounds for finite $p$--groups.}  Further investigations of $e(P)$ when $P$ is a finite $p$--group lead to some good bounds on cohomology nilpotence determined by subgroup structure.

The following monotonicity theorem at first surprised us, as it is false for arbitrary finite groups.

\begin{thm} \label{monotone thm}  Let $Q$ be a subgroup of a $p$--group $P$.  Then $e(Q) \leq e(P)$.
\end{thm}

An immediate first consequence is that the upper bound given in \thmref{d(G) thm} simplifies.

\begin{thm} \label{d(P) thm} If $P$ is a $p$--group, then $d^{\U}(P) \leq e(P)$.
\end{thm}

We then make further use of \thmref{monotone thm}.  The theorem, when combined with an explicit calculation of the $e$--invariant of the $p$--Sylow subgroups of the symmetric groups, leads to the next estimate of $e(P)$.

\begin{thm} \label{perm thm}  Suppose a $p$--group $P$ acts faithfully on a set $S$ with no fixed points. Then
\begin{equation*}
e(P) \leq
\begin{cases}
|S|/2 - |S/P| & \text{if } p = 2 \\ 2|S|/p - |S/P| & \text{if } p \text{ is odd}.
\end{cases}
\end{equation*}
\end{thm}

Here $|S|$ is the cardinality of $S$. \\

Another new general bound on $e(P)$ is the following.

\begin{thm} \label{subgroup thm}  Let $A < P$ be an abelian subgroup of maximal order in a $p$--group $P$.  Then $e(P) \leq c(P)(2|P|/|A| -1)$.
\end{thm}

\begin{ex} \label{SU3 ex}  Both of these last two theorems are nicely illustrated by the following example.  Let $P$ be a 2--Sylow subgroup of the finite group $SU(3,4)$.  $P$ is a 2--central group of order 64, of exponent 4, with $C(P)= [P,P] \simeq \Z/2 \times \Z/2$: see \cite[\S 6.3]{k4} for a useful description of this group.  Both theorems give us the estimate $e(P) \leq 14$, which, in fact, computation shows equals $e(P)$, and thus $d^{\U}(P)$.

To use \thmref{perm thm}, let $a,b \in P$ be elements of order 4 with $a^2 \neq b^2$. Then $P$ acts faithfully on $S = P/\langle a \rangle \coprod P/\langle b \rangle$ with no fixed points, so $e(P) \leq 32/2 - 2 = 14$.

To use \thmref{subgroup thm}, the centralizer of any element of order 4 is isomorphic to $\Z/4 \times \Z/4$, thus $e(P) \leq 2[2(64/16)-1] = 14$.
\end{ex}

\thmref{subgroup thm} is proved using Chern classes of representations, and would be a special case of the next conjecture, where we let $n(G)$ denote the minimal dimension (over $\mathbb C$) of a faithful complex representation of $G$.

\begin{conj} \label{chern conj} If $G$ is compact Lie, then $e(G) \leq 2n(G) - c(G)$.
\end{conj}

If this conjecture were true, one could easily deduce that $d^{\U}(G)  \leq 2n(G) - c(G)$: see \remref{rep remark}. This should be compared to the estimate in \cite{hls1}: $d^{\U}(G) \leq n(G)^2$.

Section \ref{new p group results} contains the proofs of \thmref{perm thm}, \thmref{subgroup thm}, a discussion of the conjecture, and the beginning of our most subtle argument: the proof of \thmref{monotone thm}.  Proved by induction on the order of $P$, in \secref{new p group results} it is reduced to a problem about invariants of arbitrary $\Z/p$ actions on subHopf algebras of polynomial algebras over $\F_p$: see \probref{invariant problem}.  This we then deal with in \secref{inv section}, proving results in invariant theory which appear to be new, and should be of independent interest.

\begin{rem} We note that our paper \cite{k5} has tables of examples made using the Jon Carlson's cohomology website \cite{carlson website}.  Thousands more examples are now similarly accessible using the cohomology website of David Green and Simon King \cite{david green's website}.  Their implementation includes the calculation of the restriction of $H^*(BP)$ to $H^*(BC(P))$, so that $e(P)$ can be immediately read off of their data.  For example, one see that if $P$ is the 2--Sylow subgroup of the third Conway group, so $P$ has order 1024, then $e(P) = 7$ and so, combining \thmref{d(P) thm} with the fine points of \thmref{indec < e thm}, we see that $d^{\U}(P) \leq 6$.
\end{rem}

\subsection{Acknowledgements}  The organization of \secref{old results revisited} follows the presentation I gave at the Conference on Algebraic Topology, Group Theory, and Representation Theory held on the Isle of Skye, Scotland in June, 2009. I have tried to keep the audience I had there in mind.  Conversations with Benson, Symonds, and invariant theorists David Wehlau, Jim Shank, and Eddy Campbell have been helpful.

\section{Old results revisited} \label{old results revisited}

In this section, we prove the bounds for $d^{\U}(G)$ given in \thmref{d(G) thm}.  The main steps are as follows, where terminology and notation will be defined in due course: \\
\begin{itemize}
\item $d^{\U}(G) = \max \ \{ d_0(Cess^*(BC_G(E))) \ | \ E < G \}$. See \propref{d(G) cess prop}.\\
\item $d_0(Cess^*(BG)) = e_{prim}(G)$. See \corref{d Cess = e prim cor}. \\
\item $e_{prim}(G) \leq e_{indec}(G)$.   See \corref{prim indec cor}.  \\
\item $e_{indec}(G) \leq e(G) - \dim G$.  See \thmref{indec < e thm} for this and a bit more. \\
\end{itemize}
This last inequality refines using local cohomology as follows: \\
\begin{itemize}
\item $e_{indec}(G) = e(G) + \max\{ e \ | \ H^{c(G),-c(G)+e}_{\mathfrak m}(H^*(BG)) \neq 0\}$. \\ See \thmref{e indec loc coh thm}. \\
\item $H^{s,t}_{\mathfrak m}(H^*(BG)) = 0$ if $s+t > - \dim G$.  This is Symonds' theorem \cite{symonds}. \\
\end{itemize}

\subsection{The basic ring structure of $H^*(BG)$}

We begin by recalling a fundamental example.  If $E = (\Z/p)^r$, and $H^1(E) \simeq E^{\#}$ has basis $x_1, \dots, x_r$, then
\begin{equation*}
H^*(BE) \simeq
\begin{cases}
\F_2[x_1,\dots,x_r] & \text{if $p=2$}  \\ \Lambda(x_1,\dots,x_r) \otimes \F_p[y_1,\dots,y_r] & \text{if $p$ is odd},
\end{cases}
\end{equation*}
where $y_i = \beta(x_i)$.  ($\beta$ is the Bockstein homomorphism.)  Furthermore, addition $E \times E \ra E$ induces a primitively generated Hopf algebra structure on $H^*(BE)$.

More generally, $H^*(BG)$ can be difficult to compute explicitly, particularly when $G$ is a more interesting finite $p$--group.  For example, if $P$ is the 2--Sylow subgroup of $SU_3(4)$, as in \exref{SU3 ex}, a minimal presentation of the algebra $H^*(BP)$ has 26 generators (in degrees up to 11) and 270 relations (in degrees up to 22): see \cite[group \#187]{carlson et al}, or \cite[group \#145]{david green's website}.

In spite of this, some basic ring structure has been known for a long time.  In the late 1960's \cite{quillen} D.Quillen showed that $H^*(BG)$ is Noetherian of Krull dimension $r(G)$; equivalently, $H^*(BG)$ is a finitely generated module over a polynomial subalgebra on $r(G)$ generators.   A decade later J.Duflot \cite{duflot} showed that its depth is at least $c(G)$; equivalently, $H^*(BG)$ is a free module over a polynomial subalgebra on $c(G)$ generators.

\begin{rem} The extreme situation, when $c(G) = r(G)$, happens precisely when $G$ is $p$--central.  Then $H^*(BG)$ will be Cohen--MacCauley: the depth of $H^*(BG)$ will equal its Krull dimension.  In general, there is no group theoretic criterion characterizing either groups $G$ such that the depth of $H^*(BG)$ equals the lower bound $c(G)$, or groups $G$ such that the depth of $H^*(BG)$ equals the upper bound $r(G)$.
\end{rem}

Quillen's idea was to probe $H^*(BG)$ by its restrictions to its elementary abelian $p$--subgroups.  The product over all such restrictions gives a ring homomorphism
$$ q_0: H^*(BG) \ra \prod_{E < G} H^*(BE).$$
Recall that, given $K < G$, the restriction map $H^*(BG) \ra H^*(BK)$ makes $H^*(BK)$ into a finitely generated $H^*(BG)$--module.  Thus the codomain of $q_0$, a ring whose Krull dimension is clearly $r(G)$, is finitely generated over $H^*(BG)$.  Quillen then shows that $\ker(q_0)$ is nilpotent, which then immediately implies the result about Krull dimension.

\subsection{The nilpotent filtration of $\U$}

As the mod $p$ cohomology of a topological space, $H^*(BG)$ is an unstable algebra over the mod $p$ Steenrod algebra $\A_p$.   When $p=2$, we recall that an $\A_p$--module $M$ is unstable if
$Sq^k x = 0$ whenever $k>|x|$. When $p$ is odd, the condition is that $\beta^eP^k x = 0$ if $2k+e>|x|$.  $M$ is an unstable algebra if in addition, it is a graded commutative algebra satisfying both the Cartan and Restriction formulae.

The 1980's featured much remarkable work on $\K$ and $\U$, the categories of unstable algebras and modules, with the algebras $H^*(BE)$ playing a special role. (See \cite{s2} for entry into the extensive literature.)

In the 1995 paper \cite{hls1}, H.-W.Henn, J.Lannes, and L.Schwartz  revisited Quillen's work from this new perspective.  Following \cite{hls1}, we have the following definition.

\begin{defn}  If $M$ is an unstable $\A_p$--module, let $d_0(M)$ be the maximal $d$ such that $M$ contains a nonzero submodule of the form $\Sigma^d N$, with $N$ unstable.  If no such maximum exists, let $d_0(M) = \infty$, and let $d_0(\mathbf 0) = -\infty$.
\end{defn}

Thus the invariant $d^{\U}(G)$ of the introduction is $d_0(H^*(BG))$.

An alternate definition, easily shown equivalent to the one above, is that $d_0(M)$ is the length of the nilpotent filtration \cite{s1} of $M$,
$$ \dots   \subset nil_d M \subset nil_{d-1} \subset nil_1 M \subset nil_0 M = M,$$
where $nil_dM$ is the large submodule in the localizing subcategory of $\U$ generated by the $d$--fold suspensions.

Three elementary properties of $d_0(M)$ are stated in the next lemma.

\begin{lem} {\bf (a)} \ If $M$ is nonzero in degree $d$, but zero in all higher degrees, then $d_0(M) = d$.

\noindent{\bf (b)} \ If $0 \ra M_1 \ra M_2 \ra M_3 \ra 0$ is a short exact sequence in $\U$, then $d_0(M_1) \leq d_0(M_2)$, and $d_0(M_2) \leq \max \{d_0(M_1), d_0(M_3)\}$.

\noindent{\bf (c)} \ $d_0(H^*(B\Z/p)) = 0$.
\end{lem}

The next properties are considerably deeper.  References for (a) are \cite[Prop.2.5]{k2} or \cite[Prop.I.3.6]{hls1} .  Property (b) concerns Lannes' functor \cite{L3} $T_E: \U \ra \U$, the left adjoint to the functor $M \rightsquigarrow H^*(BE) \otimes M$, and a reference is \cite[Prop.3.12]{k5}.  Property (c) is due to Henn \cite{henn}.

\begin{prop} \label{d(M) properties} {\bf (a)} \ $d_0(M \otimes N) = d_0(M) + d_0(N)$.

\noindent{\bf (b)} \   $d_0(T_EM) = d_0(M)$.

\noindent{\bf (c)} \ $d_0(M) < \infty$ if $M$ is a finitely generated module over an Noetherian unstable algebra $K$ with structure map $K \otimes M \ra M$ in $\U$.
\end{prop}

\subsection{The comparison between $d^{alg}(G)$ and $d^{\U}(G)$}

Note that property $(c)$ of the last proposition implies that $d^{\U}(G) < \infty$, so that the nilpotent filtration of $H^*(BG)$ has finite length.

In \cite{hls1}, the authors show how to generalize Quillen's map $q_0$ to realize the nilpotent filtration of $H^*(BG)$.  For each $d \geq 0$, let
$$q_d: H^*(BG) \ra  \prod_E H^*(BE) \otimes H^{\leq d}(BC_G(E))$$
be the map of unstable algebras with components induced by the the group homomorphisms $E \times C_G(E) \ra G$.  Here $M^{\leq d}$ denotes the quotient of a graded module $M$ by all elements of degree more than $d$.

They observe that $\ker q_d = nil_{d+1} H^*(BG)$, and so we have the following.

\begin{prop}  $d^{\U}(G)$ is the minimal $d$ such that $q_d$ is monic.
\end{prop}

If $I$ is a nilpotent ideal in a graded Noetherian ring, let $d^{alg}(I)$ be the maximal $d$ such that $I^d \neq 0$.  Thus the invariant $d^{alg}(G)$ of the introduction is $d^{alg}(\Rad(G))$.
Note that
\begin{equation*}
d^{alg}(H^*(BE) \otimes \Tilde H^{\leq d}(BC_G(E))) =
\begin{cases}
d^{alg}(\tilde H^{\leq d}(BC_G(E)) & \text{if } p=2 \\ d^{alg}(\Lambda(E^{\#}) \otimes \tilde H^{\leq d}(BC_G(E))) & \text{if } p \text{ is odd}.
\end{cases}
\end{equation*}

\begin{cor}  \label{alg vs module cor} With $d=d^{\U}(G)$, \begin{equation*}
d^{alg}(G) \leq
\begin{cases}
{\displaystyle \max_E} \{d^{alg}(\tilde H^{\leq d}(BC_G(E))\}  \leq d & \text{if } p=2 \\ {\displaystyle \max_E} \{d^{alg}(\tilde H^{\leq d}(BC_G(E))+ r(E)\} \leq d + r(G)& \text{if } p \text{ is odd}.
\end{cases}
\end{equation*}
\end{cor}

\subsection{Central essential cohomology}

The following definition from \cite{k5} is a variant of Carlson's Depth Essential Cohomology \cite{carlson et al}.

\begin{defn} \ Let $Cess^*(BG)$ be the kernel of the map
$$ H^*(BG) \ra \prod_{C(G)\lvertneqq E} H^*(BC_G(E)).$$
\end{defn}

This is an unstable $\A$--module. $Cess^*(BG) = H^*(BG)$ exactly when the product is over the empty set, i.e. $G$ is $p$--central.  $Cess^*(BG)$ can also be zero: as we will see, $Cess^*(BG) \neq 0$ if and only if the depth of $H^*(BG) = c(G)$.

\begin{prop} \label{d(G) cess prop}  $d^{\U}(G) = \max \ \{ d_0(Cess^*(BC_G(E))) \ | \ E < G \}$.
\end{prop}

To prove this, we first need the following consequence of the calculation of $T_EH^*(BG)$ due to Lannes \cite{lannes 2}.

\begin{prop}  For all $E < G$, $H^*(BC_G(E))$ is a summand of $T_EH^*(BG)$, and thus $d^{\U}(C_G(E)) \leq d^{\U}(G)$.
\end{prop}

\begin{proof}[Proof of \propref{d(G) cess prop}]  This follows by downward induction on the rank of $C(G)$.  From the exact sequence
$$ 0 \ra Cess^*(BG) \ra H^*(BG) \ra \prod_{C(G)\lvertneqq E} H^*(BC_G(E)),$$
one sees that
$$ d^{\U}(G) \leq \max \ \{ d_0(Cess^*(BG)), d^{\U}(C_G(E)) \ | \ C(G)\lvertneqq E < G \}.$$
But this inequality is an equality by the last proposition.
\end{proof}

\subsection{Primitives in central essential cohomology}

For the rest of this section, we fix a compact Lie group $G$, and let $C = C(G)$.

By an unstable $H^*(BC)$--comodule, we will mean an unstable module $M$ having an $H^*(BC)$--comodule structure map $\Delta: M \ra H^*(BC) \otimes M$ that is in the category $\U$.   Examples of interest to us include $H^*(BG)$, $H^*(BC_G(E))$ for all $E <G$, and $Cess^*(BG)$, where the comodule structures are all induced by the group homomorphism $C \times G \ra G$ sending $(c,g)$ to $cg$.

\begin{defns} If $M$ is an unstable $H^*(BC)$--comodule, we define its associated module of primitives to be
$$ P_CM = \{ x \in M \ | \ \Delta(x) = 1 \otimes x \} = \text{Eq } \{ M
\begin{array}{c} \Delta \\[-.08in] \longrightarrow \\[-.1in] \longrightarrow \\[-.1in] i
\end{array}
H^*(C) \otimes M \}.$$
If $P_CM$ is finite dimensional, we let $e_{prim}(M)$ be its largest nonzero degree, or $-\infty$ if $M = \mathbf 0$.
\end{defns}

Note that $P_CM$ is again an unstable module.

\begin{lem}  If $M$ is an unstable $H^*(BC)$--comodule, and $P_CM$ is finite dimensional, then $d_0(M) = e_{prim}(M).$
\end{lem}
\begin{proof} Assume $P_CM$ is finite dimensional with largest nonzero degree $e = e_{prim}(M)$.  Then $e= d_0(P_CM)$.  Since $P_CM$ is an unstable submodule of $M$, $d_0(P_CM) \leq d_0(M)$.  Finally, the composite
$$ M \xra{\Delta} H^*(BC) \otimes M \twoheadrightarrow H^*(BC) \otimes M^{\leq e}$$
will be monic, so that $$d_0(M) \leq d_0(H^*(BC) \otimes M^{\leq e}) = d_0(M^{\leq e}) = e.$$
\end{proof}

\begin{prop} \label{fin dim prim prop} $P_CCess^*(BG)$ is finite dimensional.
\end{prop}
\begin{proof}\cite[Thm.8.5]{k5} implies that if $P_CCess^d(BG) \neq 0$, then $d\leq d^{\U}(G)$.
\end{proof}

\begin{rem}  The careful reader will discover that \cite[Thm.8.5]{k5} has a rather delicate proof, using related results in \cite{k4}, all based on careful analysis of formulae in \cite{hls1}.  It would be nice to have a simpler proof of the proposition.  In the next subsection we will see (\corref{prim indec fin cor}) that $P_CCess^*(BG)$ is finite dimensional if and only if $Cess^*(BG)$ has Krull dimension equal to $c(G)$.  When $G$ is finite, this Krull dimension calculation is verified \cite[Prop.8.2]{k5} using a result of J.Carlson \cite{carlson depth}.
\end{rem}

We let $e_{prim}(G)$ denote $e_{prim}(Cess^*(BG))$.

\begin{cor} \label{d Cess = e prim cor} $d_0(Cess^*(BG)) = e_{prim}(G)$.
\end{cor}

\subsection{Duflot algebras}  Let $c = c(G)$, the rank of $C = C(G)$, so that
\begin{equation*}
H^*(BC) \simeq
\begin{cases}
\F_2[x_1,\dots,x_c] & \text{if $p=2$}  \\ \Lambda(x_1,\dots,x_c) \otimes \F_p[y_1,\dots,y_c] & \text{if $p$ is odd}.
\end{cases}
\end{equation*}

The image of the restriction homomorphism $i^*: H^*(BG) \ra H^*(BC)$ will be a sub Hopf algebra of $H^*(BC)$.  After a change of basis for $H^1(BC)$, it will have the form
\begin{equation*}
\im(i^*) =
\begin{cases}
\F_2[x_1^{2^{j_1}},\dots,x_c^{2^{j_c}}] & \text{if $p=2$}  \\

\F_p[y_1^{p^{j_1}},\dots,y_b^{p^{j_b}},y_{b+1},\dots,y_c] \otimes \Lambda(x_{b+1},\dots,x_c) & \text{if $p$ is odd},
\end{cases}
\end{equation*}
with the $j_i$ forming a sequence of nonincreasing nonnegative integers. (See \cite[Rem.1.3]{broto henn} and  \cite{aguade smith}.) In the odd prime case, $c-b$ has group theoretic meaning as the rank of the largest subgroup of $C$ splitting off $G$ as a direct summand.

As in \cite{k5}, we will say that $G$ has {\em type} $[a_1,\dots,a_c]$ where
\begin{equation*}
(a_1,\dots,a_c) =
\begin{cases}
(2^{j_1}, \dots, 2^{j_c}) & \text{if $p=2$}  \\ (2p^{j_1}, \dots, 2p^{j_b},1, \dots,1) & \text{if $p$ is odd}.
\end{cases}
\end{equation*}

Recall that $e(G)$ is defined to be the largest degree of a $H^*(BG)$--module generator of $H^*(BC)$, i.e. the top degree of the finite dimensional Hopf algebra $H^*(BC) \otimes_{H^*(BG)} \F_p$.   Note that this number is determined by the type of $G$:
$$ e(G) = \sum_{i=1}^c (a_i-1).$$

Since $\im(i^*)$ is a {\em free} commutative algebra, one can split the epimorphism of rings $i^*: H^*(BG) \twoheadrightarrow im(i^*)$, and make the next definition.

\begin{defn} A {\em Duflot algebra} of $H^*(BG)$ is a subalgebra $A \subseteq H^*(BG)$, such that $i^*: A \ra \im(i^*)$ is an isomorphism.
\end{defn}

\begin{rem}  It seems unclear that a Duflot algebra can always be chosen to also be closed under  Steenrod operations.  Nor does it seem that it can be always chosen to be a sub--$H^*(BC)$--comodule of $H^*(BG)$.
\end{rem}

\subsection{Indecomposables in central essential cohomology}

For the rest of the section, now also fix a Duflot algebra $A \subseteq H^*(BG)$.

\begin{defns} If $M$ is an $A$--module, we define the $A$--indecomposables to be $Q_AM = M \otimes_{A} \F_p$.  If $Q_AM$ is finite dimensional, we let $e_{indec}(M)$ be its largest nonzero degree, or $-\infty$ if $M = \mathbf 0$.
\end{defns}

Observe that everything in the exact sequence
\begin{equation*}0 \ra Cess^*(BG) \ra H^*(BG) \ra \prod_{C\lvertneqq E} H^*(BC_G(E))
\end{equation*}
is both an $A$--module and a $H^*(BC)$--comodule.  These structures are sufficiently compatible `up to filtration' so that one can prove the following.

\begin{prop} \label{Cess prop}  The following hold. \\

\noindent{\bf (a)}  $Cess^*(BG)$ is a free $A$--module. \\

\noindent{\bf (b)} The composite $P_C Cess^*(BG) \hra Cess^*(BG) \era Q_ACess^*(BG)$ is monic. \\

\noindent{\bf (c)} The sequence $\displaystyle 0 \ra Q_ACess^*(BG) \ra Q_AH^*(BG) \ra \prod_{C\lvertneqq E} Q_AH^*(BC_G(E))$
is exact.
\end{prop}

See \cite[Prop.8.1]{k5}.

\begin{cor} \label{prim indec fin cor} $Q_ACess^*(BG)$ is finite dimensional if and only if $P_CCess^*(BG)$ is finite dimensional.  In this case, $e_{prim}(Cess^*(BG)) \leq e_{indec}(Cess^*(BG))$.
\end{cor}
\begin{proof}  For notational simplicity, let $M = Cess^*(BG)$.  The proposition immediately implies that if $Q_AM$ is finite dimensional so is $P_CM$, and the stated inequality will hold. Conversely, suppose $P_CM$ is finite dimensional. Recall that the composite (of $A$--modules)
$$ M \xra{\Delta} H^*(BC) \otimes M \ra H^*(BC) \otimes M^{\leq e_{prim}(M)}$$
is monic.  As $H^*(BC) \otimes M^{\leq e_{prim}(M)}$ is certainly a finitely generated $A$--module, so is $M$.
\end{proof}

We let $e_{indec}(G)$ denote $e_{indec}(Cess^*(BG))$.

\begin{cor} \label{prim indec cor} $Cess^*(BG)$ is a finitely generated free $A$--module, and $e_{prim}(G) \leq e_{indec}(G)$.
\end{cor}

\begin{rem} As we observed computationally in \cite[Appendix A]{k5}, $e_{prim}(G) = e_{indec}(G)$ for all finite 2--groups $G$ of order dividing 32.  We suspect that this pattern will not continue, but it would be nice to have an explicit example for which the inequality of the corollary is strict.
\end{rem}

\subsection{Local cohomology and Symond's theorem}  The last step in our proof of \thmref{d(G) thm} is the verification of the next bound.

\begin{thm}  \label{indec < e thm} For all $G$, $e_{indec}(G) \leq e(G) - \dim G$.  The inequality is strict unless $G$ is $p$--central.  If $G$ is $p$--central, then $d_{\U}(G) = e_{prim}(G) = e_{indec}(G) = e(G) - \dim G$.
\end{thm}

We first note that, even when $p$ is odd, it suffices to prove this when the Duflot algebra $A$ is a polynomial algebra, i.e.~ when $G$ has no $\Z/p$ direct summands, as $d_{\U}(G \times E) = d_{\U}(G)$, $e_{prim}(G \times E) = e_{prim}(G)$, $e_{indec}(G \times E) = e_{indec}(G)$, and $e(G \times E) = e(G)$.

We need to begin with a quick summary of definitions and properties of local cohomology.  A general reference for this is \cite{brodmann sharp}.

Let $\mathfrak m$ be a maximal ideal in a graded Noetherian ring $R$.  For $M$ an $R$--module,
$$ M \mapsto H_{\mathfrak m}^{s,*}(M)$$
is defined to be the $s$th right derived functor of
$$ M \mapsto H^{0,*}_{\mathfrak m}(M) = \text{ the ${\mathfrak m}$--torsion part of $M$}.$$

\begin{prop}
$H_{\mathfrak m}^{s,*}(M) \neq 0$ only if $\text{depth}_{\mathfrak m}M \leq s \leq \dim M$.  Furthermore,
If $s = \text{depth}_{\mathfrak m}M$ or $s = \dim M$, then $H_{\mathfrak m}^{s,*}(M) \neq 0$.
\end{prop}

This is the content of \cite[Cor.6.2.8]{brodmann sharp}.

We need some related results about how local cohomology interacts with regular $M$--sequences. Let $|z|$ denote the degree of $z \in R$.

\begin{lem} Fix $(s,t)$, and suppose that $H_{\mathfrak m}^{s^{\prime},t^{\prime}}(M) = 0$ for $s^{\prime} < s$ and for $(s,t^{\prime})$ with $t^{\prime}>t$.  If $z \in R$ is an $M$--regular element, then $H_{\mathfrak m}^{s^{\prime},t^{\prime}}(M/(z)) = 0$ for $s^{\prime} < s-1$ and for $(s-1,t^{\prime})$ with $t^{\prime}>t+|z|$, and, furthermore
$$ H_{\mathfrak m}^{s-1,t+|z|}(M/(z)) \simeq H_{\mathfrak m}^{s,t}(M).$$
\end{lem}
\begin{proof}  By assumption, $z$ is not a zero divisor of $M$, so there is a short exact sequence of $R$--modules
$$0 \ra \Sigma^{|z|} M \xra{z} M \ra M/(z) \ra 0.$$
The lemma then follows from the associated long exact sequence, which has the form
$$ \dots \ra H_{\mathfrak m}^{s^{\prime}-1,t^{\prime}+|z|}(M) \ra H_{\mathfrak m}^{s^{\prime}-1,t^{\prime}+|z|}(M/(z)) \ra H_{\mathfrak m}^{s^{\prime},t^{\prime}}(M) \ra H_{\mathfrak m}^{s^{\prime},t^{\prime}+|z|}(M) \ra \dots$$
\end{proof}

By induction on the length of a regular sequence, the lemma has the following corollary.

\begin{cor} With assumptions on $(s,t)$ and $M$ as in the lemma, if $z_1,\dots,z_s$ is an $M$--regular sequence, then $H_{\mathfrak m}^{0,t^{\prime}}(M/(z_1,\dots,z_s)) = 0$ for $t^{\prime}>t+|z_1|+ \dots + |z_s|$, and
$$ H_{\mathfrak m}^{0,t+|z_1|+ \dots + |z_s|}(M/(z_1,\dots,z_s)) \simeq H_{\mathfrak m}^{s,t}(M).$$
\end{cor}

We now apply this in the case when $R=M=H^*(BG)$ and $\mathfrak m = \widetilde H^*(BG)$.  Let $c = c(G)$, $r = r(G)$.  If $z_1, \dots, z_c$ are algebra generators for the Duflot algebra $A$, then $|z_1|+ \dots + |z_s| = c + e(G)$, and $M/(z_1,\dots,z_c) = Q_AM$, and the corollary tells us the following.

\begin{prop} \label{shift prop} Suppose $H^{c(G),-c(G)+e^{\prime}}_{\mathfrak m}(H^*(BG)) = 0$ for all $e^{\prime}>e$. Then $H^{0,e(G)+e^{\prime}}_{\mathfrak m}(Q_AH^*(BG))= 0$ for all $e^{\prime}>e$, and $$H^{0,e(G)+e}_{\mathfrak m}(Q_AH^*(BG)) = H_{\mathfrak m}^{c,-c + e}(H^*(BG)).$$
\end{prop}

Now we note

\begin{prop}  \label{QA prop} $Q_ACess^*(BG) = H^{0,*}_{\mathfrak m}(Q_ACess^*(BG)) = H^{0,*}_{\mathfrak m}(Q_AH^*(BG))$.
\end{prop}

Our argument is similar to that proving {\cite[Prop.8.9]{k5}}.  We need

\begin{lem}[{\cite[Lem.8.8]{k5}}] Assume $c<r$.  Given any sequence $z_1, \dots, z_c \in H^*(G)$ that generates the polynomial algebra $A$, there exists $z \in H^*(BG)$ such that, for all proper inclusions $C<E$, $z_1, \dots, z_c, z$ restricts to a regular sequence in $H^*(BC_G(E))$.
\end{lem}

\begin{proof}[Proof of \propref{QA prop}]  As $Q_ACess^*(BG)$ is finite dimensional, we clearly have $$Q_ACess^*(BG) = H^{0,*}_{\mathfrak m}(Q_ACess^*(BG)).$$
By \propref{Cess prop}, we have an exact sequence
$$ 0 \ra Q_ACess^*(BG) \ra Q_AH^*(BG) \ra \prod_{C\lvertneqq E} Q_AH^*(BC_G(E)),$$
and this induces an exact sequence
$$ 0 \ra H^{0,*}_{\mathfrak m}(Q_ACess^*(BG)) \ra H^{0,*}_{\mathfrak m}(Q_AH^*(BG)) \ra \prod_{C\lvertneqq E} H^{0,*}_{\mathfrak m}(Q_AH^*(BC_G(E))).$$
But the last term here is 0, because if $z \in H^*(BG)$ is chosen as in the lemma, then $z$ will act regularly on each $Q_AH^*(BC_G(E))$ with $C(G)\lvertneqq E$.
\end{proof}

The last two propositions combine to prove the next theorem.

\begin{thm}  \label{e indec loc coh thm}  $e_{indec}(G) = e(G) + \max\{ e \ | \ H^{c(G),-c(G)+e}_{\mathfrak m}(H^*(BG)) \neq 0\}$.
\end{thm}

\begin{proof}[Proof of \thmref{indec < e thm}]  Symonds \cite{symonds} has proved that $$H^{s,t}_{\mathfrak m}(H^*(BG)) = 0 \text{ if } s+t > - \dim G.$$  Combined with the last theorem, this immediately implies the first part of the theorem: for all compact Lie group $G$,
$$e_{indec}(G) \leq e(G) - \dim G.$$
Furthermore, this inequality will be strict if and only if
$$H^{c(G),-c(G)-\dim(G)}_{\mathfrak m}(H^*(BG)) = 0.$$

To deduce more, we need to recall why Symond's result (in the finite group case) had been conjectured by Benson.  As constructed by J.P.C.Greenlees and Benson \cite{benson greenlees}, there is a spectral sequence
$$ H^{s,t}_{\mathfrak m}(H^*(BG)) = E_2^{s,t} \Rightarrow \tilde H_{-s-t}(EG_+ \sm_G S^{Ad(G)};\F_p),$$
where $S^{Ad(G)}$ is the one point compactification of the adjoint representation, so Benson was conjecturing that some evident vanishing at the level of $E_{\infty}$ happened already at $E_2$.

By Symonds' theorem, the group $H^{c(G),-c(G)-\dim(G)}_{\mathfrak m}(H^*(BG))$ consists of permanent cycles, as the differentials off of this group will take values in groups that are zero.  As this group is certainly not in the image of nonzero boundary maps, it will thus be a quotient of
$$ \tilde H_{\dim(G)}(EG_+ \sm_G S^{Ad(G)};\F_p) \simeq \begin{cases}
\F_p & \text{if } Ad(G) \text{ is $\F_p$--oriented.} \\ 0 & \text{if not.}
\end{cases}$$

In the oriented case, $H^{r(G),-r(G)-\dim(G)}_{\mathfrak m}(H^*(BG)) \simeq \F_p$, by a  generalization to all compact Lie groups of Benson's argument \cite{benson nyjournal} in the finite group case.  (The generalization is straightforward, using the transfer map $H^*(BE) \ra H^{*+\dim(G)}(BG)$ associated to an inclusion $E < G$.)

Thus, in either the oriented or nonoriented case, we see that
$$H^{c(G),-c(G)-\dim(G)}_{\mathfrak m}(H^*(BG)) = 0$$ unless $c(P) = r(P)$, i.e. $G$ is $p$--central.  In the $p$--central case, $G$ will be oriented and $e_{indec}(G) = e(G)$.  But arguing as in \cite{k5}, one can do better: the top class in $Q_AH^*(BG)$ will be represented by a $H^*(BC)$--primitive, so $e_{prim}(G) = e_{indec}(G)$.
\end{proof}

We end this section by noting that our results above include a proof of Carlson's Depth Conjecture in the case of minimal depth, generalizing results in \cite{green,k5}.   Note that
$$e_{indec}(G) \neq -\infty \Leftrightarrow e_{indec}(G) \geq 0 \Leftrightarrow Q_ACess^*(G) \neq 0 \Leftrightarrow Cess^*(G) \neq 0,$$ and
$$H^{c(G),*}_{\mathfrak m}(H^*(BG)) \neq 0 \Leftrightarrow H^*(BG) \text{ has depth precisely } c(G).$$  Therefore, \thmref{e indec loc coh thm} tells us most of the following, and Symond's theorem tells us the rest.

\begin{thm} For $G$ compact Lie, $H^*(BG)$ has depth precisely $c(G)$ if and only if $H^*(BG)$ is not detected by restriction to the cohomology rings $H^*(BC_G(E))$ for $E < G$ of rank greater than $c(G)$.  In this case, $H^{c(G),t}_{\mathfrak m}(H^*(BG)) \neq 0$ for some $-c(G) - \dim(G) \leq t\leq -c(G)-e(G)$.
\end{thm}

\begin{cor} \ If $G$ is compact Lie, and $e(G) < \dim(G)$ then $H^*(BG)$ has depth greater than $c(G)$ and is detected by restriction to the cohomology rings $H^*(BC_G(E))$ for $E < G$ of rank greater than $c(G)$.
\end{cor}

\section{New results for finite $p$ groups} \label{new p group results}

We now prove various new results about $e(P)$ when $P$ is a finite $p$--group.  We begin with a proof of \thmref{subgroup thm}, with part of the discussion relevant for all compact Lie groups $G$.  We will next deduce \thmref{d(P) thm} and \thmref{perm thm} assuming \thmref{monotone thm}.  Finally we will reduce \thmref{monotone thm} to a problem in invariant theory, to be solved in the subsequent section.

\subsection{Upper bounds for $e(P)$ coming from Chern classes.}
We use Chern classes of representation to get group theoretic upper bounds for $e(P)$ when $P$ is a finite $p$--group.  With $C=C(P)$, we need to get a lower bound on $\im(i^*)$, the image of restriction
$$ i^*: H^*(BP) \ra H^*(BC).$$

To set up notation and unify exposition, let $c = c(P)$ and let
\begin{equation*}
H^*(BC) \simeq
\begin{cases}
\F_2[x_1,\dots,x_c] & \text{if $p=2$}  \\ \Lambda(x_1,\dots,x_c) \otimes \F_p[y_1,\dots,y_c] & \text{if $p$ is odd},
\end{cases}
\end{equation*}
where $y_i \in H^2(BC)$ denotes  $\beta(x_i)$ for all primes (so that $y_i = x_i^2$ when $p=2$).  Note that each element $y_i$ is the Chern class of a unique one dimensional complex representation $\omega_i$ of $C$.

Now let $A < P$ be a maximal abelian subgroup, so that $A$ certainly contains $C$.  Each $\omega_i$ extends, possibly nonuniquely, to a one dimensional representation $\tilde \omega_i$ of $A$.  Now let $\rho_i = \Ind_A^P(\tilde \omega_i)$, a representation of $P$ of dimension $[P:A] = |P|/|A|$.

By construction, the restriction of $\rho_i$ to $C$ will be $|P|/|A|\omega_i$, which has top Chern class $y_i^{|P|/|A|}$.  We have proved the next theorem, a precise form of \thmref{subgroup thm}.

\begin{thm}  The Hopf algebra $\im(i^*)$ contains $\F_p[y_1^{|P|/|A|}, \dots, y_c^{|P|/|A|}]$.  Thus $e(P) \leq c(P) (2|P|/|A| - 1)$.
\end{thm}

\begin{rem}  Let $e_{grp}(P) = c(P) (2|P|/|A| - 1)$, where $A < P$ is an abelian subgroup of maximal order; thus the theorem says that $e(P) \leq e_{grp}(P)$.   With arguments similar, but simpler, to ones we will use in the proof of \thmref{monotone thm}, it is not hard to prove that this invariant of $p$--groups has the following monotonicity property:
$$ \text{if } Q < P, \text{ then } e_{grp}(Q)\leq e_{grp}(P).$$  This property suffices to deduce that if $P$ is a finite $p$--group, then $d^{\U}(P) \leq e_{grp}(P)$:
$$ d^{\U}(P) \leq \max_{E<G} \{ e(C_G(E)) \} \leq \max_{E<G} \{ e_{grp}(C_G(E)) \} \leq e_{grp}(P).$$
\end{rem}

\subsection{Conjectural upper bounds for $e(G)$ coming from Chern classes.} \label{chern conj section}

We continue in the spirit of the last subsection, and discuss how one might use Chern classes to prove \conjref{chern conj}.  This says that, if $n(G)$ is the minimal dimension of a faithful representation of a compact Lie group $G$, then $e(G) \leq 2n(G) - c(G)$.

With $C = C(G)$, the calculation of $e(G)$ requires understanding of the Hopf algebra $\im(i^*)$, the image of the restriction
$$ i^*: H^*(BG) \ra H^*(BC).$$

\begin{defns}{\bf (a)}\ If $\rho$ is a representation of $C$, let $\mathcal H(\rho) \subset H^*(BC)$ be the smallest Hopf algebra containing its Chern classes. \\
{\bf (b)} When $\rho$ is faithful, so can be viewed as an inclusion of $C$ into a unitary group $U$, $\mathcal H(\rho)$ will contain the image of $H^*(BU) \ra H^*(BC)$, and thus $H^*(BC)$ will be a finitely generated $\mathcal H(\rho)$ module. In this case, let $e(\rho)$ be the top degree of $Q_{\mathcal H(\rho)}H^*(BC)$. \\
{\bf (c)} \  If $G$ is a compact Lie group with $C=C(G)$,  let $\mathcal H(G)$ be the smallest Hopf algebra containing all of the $\mathcal H(\rho)$, where $\rho$ ranges over all representations of $G$, restricted to $C$, and let $e_{rep}(G)$ be the top degree of $Q_{\mathcal H(G)}H^*(BC)$.
\end{defns}

It is clear that for any representation $\rho$ of $G$, $$ \mathcal H(\rho) \subseteq \mathcal H(G) \subseteq \im(i^*),$$
so we learn the following.

\begin{prop} $e(G) \leq e_{rep}(G) \leq e(\rho)$.
\end{prop}

Thus \conjref{chern conj} would follow immediately from the next conjecture, which just concerns Chern classes of representations of elementary abelian groups.

\begin{conj} \label{e(rho) conj} Let $C$ be an elementary abelian $p$--group of rank $c$.  If $\rho$ is a faithful $n$ dimensional complex representation of $C$, then $e(\rho) \leq 2n - c$.
\end{conj}

In turn, this conjecture would be consequence of a conjectural identification of the Hopf algebra $\mathcal H(\rho)$.  To describe this, and for later purposes, we digress to describe a natural parametrization of  the sub-Hopf algebras of a polynomial algebra.

\subsection{Sub-Hopf algebras of a polynomial algebra}

Let $S^*(V)$ be the symmetric algebra generated by a $\F_p$--vector space $V$.  If $y_1, \dots, y_c$ form a basis for $V$, then $S^*(V) = \F_p[y_1,\dots,y_c]$. We describe a natural parametrization of the {\em full} sub-Hopf algebras of $S^*(V)$: sub-Hopf algebras $\HH \subseteq S^*(V)$ having Krull dimension $c$.

We need the following notation: given a subspace $W < V$, $W^{(k)} \subset S^{p^k}(V)$ denotes the span of the $p^k$th powers of the elements in $W$.

\begin{defn} Suppose $\FF$ is a finite filtration of the $\F_p$--vector space $V$:
$$ V(0) \subseteq V(1) \subseteq \dots \subseteq V(n) = V.$$
Let $\HH(\FF) \subseteq S^*(V)$ be the Hopf algebra
$$ \mathcal H(\FF) = S^*(V(0) + V(1)^{(1)} + \dots + V(n)^{(n)}).$$
\end{defn}

\begin{prop} Filtrations of $V$ correspond bijectively to the full sub-Hopf algebras of $S^*(V)$, under the correspondence $\FF \rightsquigarrow \HH(\FF)$.
\end{prop}
\begin{proof}[Sketch proof]  If $\HH \subseteq S^*(V)$ is a full sub-Hopf algebra, then there is a basis $y_1, \dots, y_c$ of $V$, and natural numbers $j_1, \dots, j_c$, such that $\HH = \F_p[y_1^{p^{j_1}}, \dots, y_c^{p^{j_c}}]$.  Then $\HH = \HH(\FF)$, where the filtration $\FF$ of $V$ has $k$th subspace $V(k)$ equal to the span of the $y_i$ satisfying $j_i\leq k$.  More intrinsically, $V(k)^{(k)} = \HH \cap V^{(k)}$.
\end{proof}

\begin{defn}  If $\FF$ is the filtration $V(0) \subseteq \dots \subseteq V(n) = V$, let
$$ e(\FF) = \sum_{k = 0}^n c_k(\FF)(2p^k-1),$$
where $c_k(\FF)$ is the rank of $V(k)/V(k-1)$.
\end{defn}

With this definition, if $C$ is an elementary abelian $p$--group, $V = \beta(H^1(BC)) \subseteq H^2(BC)$, and $\FF$ is a filtration of $V$, then $e(\FF)$ is the top degree of a generator of $H^*(BC)$, viewed as a $\HH(\FF)$--module.

We now return to our discussion of \conjref{e(rho) conj}.  So suppose $\rho$ is a faithful $n$ dimensional complex representation of $C$, where $C$ has rank $c$.  This will be a sum of line bundles, possibly with multiplicities, and so will correspond to the following data:
\begin{itemize}
\item A finite set of distinct elements $v_1, \dots, v_m \in V$ which span $V$.
\item Multiplicities $n_1, \dots, n_m \in \mathbb N$ such that $n_1 + \dots + n_m = n$.
\end{itemize}
From this data, we define a filtration $\FF_{\rho}$ of $V$ by letting $V(k)$ be the span of the $v_j$ such that $p^{k+1}$ does not divide $n_j$.

\begin{lem}$\HH(\rho) \subseteq \HH(\FF_{\rho})$.
\end{lem}

\begin{proof}
Let $ch(\rho)$ denote the total Chern class.  We will have
$$ ch(\rho) = \prod_{j=1}^m (1+v_j)^{n_j} = \prod_k \prod_{v_j \in V(k)-V(k-1)} (1+v_j^{p^k})^{n_j/p^k}.$$
As $v_j^{p^k} \in \HH(\FF_{\rho})$ for $v_j \in V(k)-V(k-1)$, we see that all the homogenous components of $ch(\rho)$ are in $\HH(\FF_{\rho})$ as well.
\end{proof}

We conjecture equality in the last lemma.

\begin{conj} \label{rho conj} $\HH(\rho) = \HH(\FF_{\rho})$.
\end{conj}

As the estimate $e(\FF_{\rho})\leq 2n-c$ is not hard to check, this conjecture implies \conjref{e(rho) conj}, and thus \conjref{chern conj}.

\begin{rem} \label{rep remark} Note that, for any $E < G$, $n(C_G(E)) \leq n(G)$ and $c(C_G(E)) \geq c(G)$.  Thus, if \conjref{chern conj} were true, we could deduce
$$  d^{\U}(G) \leq \max_{E<G} \{ e(C_G(E)) \} \leq \max_{E<G} \{ 2n(C_G(E))-c(C_G(E)) \} \leq 2n(G)-c(G).$$
\end{rem}

\subsection{Proofs of \thmref{d(P) thm} and \thmref{perm thm} assuming \thmref{monotone thm}.}

Here we assume \thmref{monotone thm}, which says that if $P$ is a $p$-group, and $Q<P$, then $e(Q) \leq e(P)$, and deduce \thmref{d(P) thm} and \thmref{perm thm}.

\begin{proof}[Proof of \thmref{d(P) thm}]  This is immediate:
$\displaystyle  d^{\U}(P) \leq \max_{E<P} \{ e(C_P(E)) \} \leq e(P)$.
\end{proof}

\begin{proof}[Proof of \thmref{perm thm}]  Suppose a $p$--group $P$ acts faithfully on a set $S$ with no fixed points. We wish to show that $e(P) \leq |S|/2 - |S/P|$ when $p=2$, and $e(P) \leq 2|S|/p - |S/P|$ when $p$ is odd.

Note that $S/P$ is the set of orbits of $S$, so $S$ has a decomposition into orbits
$$ S = \coprod_{i=1}^{|S/P|} S_i,$$
with $|S_i| = p^{r_i}$, and each $r_i\geq 1$.  Then $P$ admits an embedding
$$P \subseteq \prod_{i=1}^{|S/P|} W(r_i),$$
where $W(r)$ denotes the Sylow subgroup of the symmetric group $\Sigma_{p^r}$.

Assuming \thmref{monotone thm}, we would then have the bound
$$ e(P) \leq \sum_{i=1}^{|S/P|} e(W(r_i)).$$
The next proposition will thus complete the proof of \thmref{perm thm}
\end{proof}

\begin{prop} \label{e(W) prop} When $p=2$, $e(W(r)) = 2^{r-1} - 1$.  When $p$ is odd, $e(W(1)) = 0$, and, for $r\geq 1$, then $e(W(r)) = 2p^{r-1}-1$.
\end{prop}
\begin{proof}  We begin by identifying $C(r) = C(W(r))$.  We claim that $C(r) \simeq \Z/p$.  This is easily proved by induction on $r$, as $W(r+1)$ is the semidirect product
$$ W(r+1) = W(r)^p \rtimes \Z/p.$$
so that
$$C(r+1) = (C(r)^p)^{\Z/p},$$
the diagonal copy of $C(r)$ in $C(r)^p$.

Now we determine $\im(i(r)^*) \subset H^*(BC(r))$, where $i(r): C(r) \ra W(r)$ is the inclusion.

The case when $r=1$ is elementary: $C(1) = W(1) = \Z/p$, so $\im(i(1)^*) = H^*(B\Z/p)$ and  $e(W(1)) = 0$ for all primes $p$.

To proceed by induction, we observe that the inclusions
$$ C(r+1) \ra C(r)^p \ra W(r)^p \ra W(r+1)$$
induce a factorization of $i(r+1)^*$ as
$$ H^*(BW(r+1)) \twoheadrightarrow H^*(BW(r)^p)^{\Z/p} \xra{{(i(r)^p)}^*} H^*(BC(r)^p)^{\Z/p} \ra H^*(BC(r+1)),$$
with the first map epic as indicated.

Now let $p$ be odd.  Identifying $H^*(BC(r))$ with $\Lambda(x) \otimes \F_p[y]$, we prove by induction that, for $r \geq 2$,  $\im(i(r)^*) = \F_p[y^{p^{r-1}}]$ so that $e(W(r)) = 2p^{r-1}-1$.

The case when $r=2$ is slightly special: $\im(i(2)^*)$ will be the image of
$$  (\Lambda(x_1,\dots,x_p) \otimes \F_p[y_1,\dots,y_p])^{\Z/p} \ra \Lambda(x) \otimes \F_p[y]$$
under the map induced by sending each $x_i$ to $x$ and $y_i$ to $y$.  Recall also that this image will be a Hopf algebra.  As $y^p$ is the image of the invariant $y_1\cdots y_p$, while $x$ and $y$ are easily checked to not be in this image, we see that $\im(i(2)^*) = \F_p[y^p]$.

Assume by induction that $\im(i(r)^*) = \F_p[y^{p^{r-1}}]$.  Then, reasoning as above,
$$\im(i(r+1)^*) = \im\{ \F_p[y_1^{p^{r-1}}, \dots, y_p^{p^{r-1}}]^{\Z/p} \ra \F_p[y]\} = \F_p[y^{p^{r}}].$$

The case when $p=2$ is similar. Identifying $H^*(BC(r))$ with $\F_2[x]$, one proves by induction that, for $r \geq 1$,  $\im(i(r)^*) = \F_2[x^{2^{r-1}}]$ so that $e(W(r)) = 2^{r-1}-1$.

\end{proof}

\subsection{Reduction of \thmref{monotone thm} to invariant theory}

We begin the proof of \thmref{monotone thm}.  Our goal is to show that, if $Q$ is a subgroup of a $p$--group $P$, then $e(Q) \leq e(P)$.  Thus we need to somehow compare the image of the restriction
$$ H^*(BP) \ra H^*(BC(P))$$
to the image of the restriction
$$ H^*(BQ) \ra H^*(BC(Q)).$$

We make some first reductions.

First of all, by induction of the index of $Q$ in $P$, we can assume that $Q$ has index $p$, and thus will be normal in $P$. Then $\Z/p \simeq P/Q$ will act on $H^*(BQ)$ and also on $C(Q)$, with $C(Q)^{\Z/p} = C(P) \cap Q$.

Next, suppose that $C(P)$ is not contained in $Q$.  Then there would exist a central element $\sigma \in P$ of order $p$, not in $Q$.  It follows easily that then $\langle \sigma \rangle \times Q = P$, and we conclude that  $e(P) = e(Q)$.

Thus we will assume that $C(P)$ is contained in $Q$.  Suppose $P$ admits a direct product decomposition $P = \langle \sigma \rangle \times P_1$, with $\sigma$ of order $p$.  Then $\sigma$ would be contained in $C(P)$ and thus $Q = \langle \sigma \rangle \times Q_1$ with $Q_1 = P_1 \cap Q$.  Then $e(P) = e(P_1)$ and $e(Q) = e(Q_1)$.

We are reduced to needing to prove that $e(Q) \leq e(P)$ under the following assumptions:
\begin{itemize}
\item $Q$ is normal of index $p$, so $\Z/p \simeq P/Q$ acts on both $H^*(BQ)$ and $C = C(Q)$.
\item  $C(P) = C^{\Z/p}$.
\item $P$ has no nontrivial elementary abelian direct summands.
\end{itemize}

In this situation, the restriction map $H^*(BP) \ra H^*(C(P))$ factors
$$ H^*(BP) \ra H^*(BQ)^{\Z/p} \ra H^*(BC)^{\Z/p} \hra H^*(BC) \era H^*(BC^{\Z/p}),$$
and the last assumption tell us that the image lands in the part of $H^*(BC^{\Z/p})$ generated by $\beta(H^1(BC^{\Z/p}))$.

Let $V$ denote $\beta(H^1(BC)) \subseteq H^2(BC)$. As $V$ is naturally isomorphic to the dual of $C$, it can be viewed as a $\Z/p$--module.  Let $V_{\Z/p}$ denote the $\Z/p$--coinvariants $V/\langle x - \sigma x :  x\in V \rangle$, where $\sigma$ generates $\Z/p$.  The part of $H^*(BC^{\Z/p})$ generated by $\beta(H^1(BC^{\Z/p}))$ identifies with $S^*(V_{\Z/p})$.

As the image of $H^*(BQ) \ra H^*(BC) \era S^*(V)$ is a Hopf algebra, it must be the Hopf algebra $\HH(\FF)$ associated to a filtration $\FF$ of $V$, and $e(Q) \leq e(\FF)$.

As the map $H^*(BQ) \ra S^*(V)$ is $\Z/p$--equivariant, $\HH(\FF)$ is a sub-$\Z/p$-module of $S^*(V)$.  It follows that the filtration $\FF$ will be preserved by the $\Z/p$ action on $V$.

From our observations above, the image of $H^*(BP) \ra H^*(C(P))$ will be contained in the image of
$$ \HH(\FF)^{\Z/p} \hra S^*(V)^{\Z/p} \hra S^*(V) \era S^*(V_{\Z/p}).$$

As $e(Q) \leq e(\FF)$, we will be able to deduce that $e(Q) \leq e(P)$ if we can solve the following problem in invariant theory.

\begin{prob} \label{invariant problem} Given a filtration $\FF$ of a $\Z/p$--module $V$, find a filtration $\FF_{\Z/p}$ of $V_{\Z/p}$ such that
\begin{itemize}
\item  The image of $\HH(\FF)^{\Z/p} \ra S^*(V_{\Z/p})$ is contained in $\HH(\FF_{\Z/p})$, and
\item $e(\FF) \leq e(\FF_{\Z/p}).$
\end{itemize}
\end{prob}

In the next section we find such a filtration $\FF_{\Z/p}$: see \thmref{filtered inv thm}.

\section{New results in invariant theory} \label{inv section}

In this section $\FF$ is a filtration of an $\F_p[\Z/p]$--module $V$,
$$V(0) \subseteq V(1) \subseteq \dots \subseteq V(n) = V,$$
and we wish to understand the image of the composite
$$ \HH(\FF)^{\Z/p} \hra S^*(V)^{\Z/p} \hra S^*(V) \era S^*(V_{\Z/p}),$$
with our goal to solve \probref{invariant problem}.
Throughout we let $\sigma$ be a generator for $\Z/p$.

\subsection{$\Z/p$--modules}

The modular representation theory of $\Z/p$ is quite tame.  There are $p$ indecomposable $\F_p[\Z/p]$--modules, $V_1, \dots, V_p$, where $V_i$ has dimension $i$.  An explicit model for $V_i$ is the vector space with basis $x_1,\dots,x_i$ with
\begin{equation*}
\sigma x_j =
\begin{cases}
x_j + x_{j-1} & \text{if } 1 < j \leq i \\ x_1 & \text{if } j=1.
\end{cases}
\end{equation*}
A general  $\F_p[\Z/p]$--module $V$ decomposes as a direct sum
$$ V \simeq m_1V_1 \oplus m_2 V_2 \oplus \dots \oplus m_p V_p.$$
We say that $V$ is {\em trivial free} if $m_1 = 0$.

We let $rad(V)$ and $soc(V)$ be the radical and socle of a module $V$.  Thus $soc(V) = V^{\Z/p}$ and $V/rad(V) = V_{\Z/p}$.  In the usual way, we define $soc(V) \subset soc^2(V) \subset \dots$ and $rad(V) \supset rad^2(V) \supset \dots$.

The submodule $m_1V_1$ in a decomposition of $V$ can be regarded as the image of a section of the quotient map $soc(V) \era (soc(V) + rad(V))/rad(V)$. Thus $V$ is trivial free precisely when $soc(V) \subset rad(V)$, or equivalently, when the composite $V^{\Z/p} \hra V \era V_{\Z/p}$ is zero.

\subsection{The case when the filtration is trivial}

Given a $\Z/p$--module $V$, a special case of our general problem is to understand the image of
$$ S^*(V)^{\Z/p} \hra S^*(V) \era S^*(V_{\Z/p}).$$

We remark that, in spite of the simple classification of modules $V$, a complete calculation of  $S^*(V)^{\Z/p}$ is not known in all cases, and is the subject of much research.  Even so, we prove the following theorem.

\begin{thm} \label{inv thm}  If $V = W \oplus U$, where $W$ is trivial and $U$ is trivial free, the image of $S^*(V)^{\Z/p} \ra S^*(V_{\Z/p})$ is $S^*(W \oplus U_{\Z/p}^{(1)})$.
\end{thm}

Here is a more invariant way of stating this.  Given $V$, let $W_{\Z/p}$ be the image of the composite $V^{\Z/p} \hra V \era V_{\Z/p}$.  Then the image of
$$ S^*(V)^{\Z/p} \hra S^*(V) \era S^*(V_{\Z/p})$$
will be
$$ S^*(W_{\Z/p} + V_{\Z/p}^{(1)}).$$

The next example both illustrates the theorem and will be used in its proof.

\begin{ex} Suppose $V = mV_2$, where the $i$th copy of $V_2$ has basis $\{x_i,y_i\}$ with $\sigma y_i = y_i + x_i$ and $\sigma x_i = x_i$. The kernel of the quotient $V \era V_{\Z/p}$ is the span of the $x_i$'s, so we can view $V_{\Z/p}$ as having basis given by the $y_i$'s.  The theorem in this case is asserting that the image of the composite
$$ \F_p[x_1, \dots, x_m, y_1,\dots,y_m]^{\Z/p} \hookrightarrow \F_p[x_1, \dots, x_m, y_1,\dots,y_m]\twoheadrightarrow \F_p[y_1,\dots,y_m]$$
is $\F_p[y_1^p,\dots,y_m^p]$.  The main theorem of \cite{campbell hughes} is a description of generators of $S^*(mV_2)^{\Z/p}$ as polynomials in the $x_i$'s and $y_j$'s; see also \cite{wehlau}. One sees that all of these are sent to 0 modulo the ideal $(x_1,\dots,x_m)$ except for the `norm' generators $\prod_{j=0}^{p-1} \sigma^j y_i = y_i^p - x_i^{p-1}y_i$, which map to $y_i^p$.  So the assertion of the theorem is true in this case.
\end{ex}

\begin{proof}[Proof of \thmref{inv thm}]  First we note that if $V = W \oplus U$ with $W$ trivial, then $S^*(V)^{\Z/p} = S^*(W) \otimes S^*(U)^{\Z/p}$ and $S^*(W \oplus U_{\Z/p}^{(1)}) = S^*(W) \otimes S^*(U_{\Z/p}^{(1)})$.

Thus it suffices to prove that, when $V$ is trivial free, there is an equality
$$  I(V) = S^*(V_{\Z/p}^{(1)}),$$
where $I(V) = \im\{S^*(V)^{\Z/p} \hra S^*(V) \era S^*(V_{\Z/p})\}$.

The previous example showed that this holds when $V = mV_2$.  We use this to show that the equality holds for a general trivial free $V$.  Recall that $V_{\Z/p} = V/rad(V)$. If we let $\bar V = V/rad^2(V)$, and let $\widetilde V$ be the projective cover of $V$, then $\bar V = mV_2$ and $\widetilde V = mV_p$, if $V_{\Z/p} = mV_1$.  The surjections
$\widetilde V \era V \era \bar V$
will induce isomorphisms
$ \widetilde V_{\Z/p} = V_{\Z/p} = \bar V_{\Z/p}$, and then inclusions
$$ I(\widetilde V) \subseteq I(V) \subseteq I(\bar V) = S^*(V_{\Z/p}^{(1)}).$$
Finally, to see that all of these inclusions are, in fact, equalities, we note $I(\widetilde V_p)$ is easily seen to contain $S^*(V_{\Z/p}^{(1)})$: our proof of \propref{e(W) prop} showed that $I(V_p) = S^*(V_1^{(1)})$, and so $I(mV_p)$ certainly contains $S^*(mV_1^{(1)})$.
\end{proof}

\subsection{The case when the filtration is non-trivial}

Now suppose that there exists a decomposition of {\em filtered} $\Z/p$--modules $V = W \oplus U$, with $W$ trivial and $U$ trivial free.  Define a filtration $\FF_{\Z/p}$ of $V_{\Z/p}$ by letting
$$V_{\Z/p}(k) = (W(k) + U(k-1) + rad(V))/rad(V).$$

\begin{prop} \label{straight inv prop} The image of $\HH(\FF)^{\Z/p} \ra S^*(V_{\Z/p})$ is contained in $\HH(\FF_{\Z/p})$.
\end{prop}
\begin{proof}  Just as in the proof of \thmref{inv thm}, it suffices to prove this when $V$ is trivial free, and then $\FF_{\Z/p}$ is defined by the simpler formula
$$V_{\Z/p}(k) = (V(k-1) + rad(V))/rad(V).$$

Also, similar to the proof of \thmref{inv thm}, we let $\bar V = V/rad^2(V)$, with filtration $\bar \FF$ defined by $\bar V(k) = (V(k) + rad^2(V))/rad^2(V)$.  Then
$$ \im \{\HH(\FF)^{\Z/p} \ra S^*(V_{\Z/p}) \} \subseteq \im \{\HH(\bar \FF)^{\Z/p} \ra S^*(V_{\Z/p})\},$$ and the filtrations $\FF_{\Z/p}$ and  $\bar \FF_{\Z/p}$ of $V_{\Z/p}$ agree.
Thus it suffices to also assume that $V$ satisfies $rad^2(V) = 0$, so that $V$ is isomorphic to $mV_2$ for some $m$.

In this case, let elements $y_1, \dots, y_m \in V$, of filtration $k_1, \dots, k_m$, project to a filtered basis of $V_{\Z/p}$, and let $x_j = \sigma y_j - y_j$.  Then
$$ \HH(\FF) \subseteq \F_p[x_1, \dots, x_m, y_1^{p^{k_1}}, \dots, y_m^{p^{k_m}}]$$
as algebras with $\Z/p$ action, and so the image of $\HH(\FF)^{\Z/p} \ra S^*(V_{\Z/p})$ is contained in $\HH(\FF_{\Z/p}) = \F_p[y_1^{p^{k_1}+1}, \dots, y_m^{p^{k_m}+1}]$, as this is the image of
$$ \F_p[x_1, \dots, x_m, y_1^{p^{k_1}}, \dots, y_m^{p^{k_m}}]^{\Z/p} \ra \F_p[y_1,\dots,y_m].$$
\end{proof}

\subsection{The general case}

Unfortunately, at least when $p\geq 3$, a general filtered $\Z/p$--module $V$ need not admit a direct sum decomposition as {\em filtered} modules of the form $V = W \oplus U$, with $W$ trivial and $U$ trivial free.

\begin{ex}  \label{funny example} With $p\geq 3$, let $V(0) = V_2$ embedding `diagonally' in $V_1 \oplus V_3 = V(1) = V$.  Then the image of $soc(V) \ra V/rad(V)$ is $V_1$, generated by an element of $V(0)$, but not of $soc(V(0))$, and we see that there is no isomorphism $V \simeq V_1 \oplus V_3$ as filtered modules\footnote{We thank Dave Benson for showing us this example.}.
\end{ex}

This phenomenon goes away if we assume that $rad(V) \subseteq V(0)$.

\begin{lem} \label{decomp lemma} If $rad(V) \subseteq V(0)$, then there exists a decomposition of filtered $\Z/p$--modules $V = W \oplus U$, with $W$ trivial and $U$ trivial free.
\end{lem}

We temporarily postpone the proof. \\

Now let $\FF$ be an arbitrary filtration of a $\Z/p$--module $V$.  Define a filtration $\FF_{\Z/p}$ of $V_{\Z/p}$ by letting
$$ V_{\Z/p}(k) = (soc(V(k) + rad(V)) + V(k-1) + rad(V))/rad(V).$$
Note that, if $V = W \oplus U$ with $W$ trivial and $U$ reduced, then the filtration $\FF_{\Z/p}$ agrees with the filtration of the same name in the last subsection.

The next theorem says that this filtration solves \probref{invariant problem}.

\begin{thm}  \label{filtered inv thm} {\bf (a)} The image of $\HH(\FF)^{\Z/p} \ra S^*(V_{\Z/p})$ is contained in $\HH(\FF_{\Z/p})$. \\
\noindent{\bf (b)} $e(\FF) \leq e(\FF_{\Z/p})$.
\end{thm}

\begin{proof}[Proof of \thmref{filtered inv thm}(a)]  Define a new $\Z/p$--equivariant filtration $\FF^{\prime}$ of $V$ by letting
$$ V^{\prime}(k) = V(k) + rad(V),$$
so that, for all $k$,
$$ V(k) \subseteq V^{\prime}(k).$$
Then $\HH(\FF) \subseteq \HH(\FF^{\prime})$, so that
$$ \im\{\HH(\FF)^{\Z/p} \ra S^*(V_{\Z/p}) \} \subseteq \im \{\HH(\FF^{\prime})^{\Z/p} \ra S^*(V_{\Z/p}) \}.$$
Moreover, $\FF_{\Z/p}^{\prime} = \FF_{\Z/p}$.

By construction, $rad(V) \subseteq V^{\prime}(0)$, and so \lemref{decomp lemma} applies.  Thus part (a) of the theorem follows from \propref{straight inv prop} which tells us that
$$\im \{\HH(\FF^{\prime})^{\Z/p} \ra S^*(V_{\Z/p}) \} \subseteq \HH(\FF_{\Z/p}).$$
\end{proof}
\begin{proof}[Proof of \thmref{filtered inv thm}(b)]  For $v \in V$, let $\bar v$ denote its image in $V_{\Z/p}$.  We define $|v| = k$ if $v \in V(k)-V(k-1)$, and similarly define $|\bar v|$.  We say that a basis $\{v_{\alpha}\}$ for $V$ is a filtered basis if, for all $k$, $\{ v_{\alpha} \ | \ |v_{\alpha}| \leq k\}$ is a basis for $V(k)$.

One can choose a filtered basis for $V$ which include families of elements $y_{\beta}, z_{\gamma}$ such that the $\bar z_{\gamma}$ form a basis for $(soc(V) + rad(V))/rad(V)$, and the $\bar y_{\beta}, \bar z_{\gamma}$ form a basis for $V_{\Z/p} = V/rad(V)$.  Then $|\bar y_{\beta}| = |y_{\beta}|+1$, while $|\bar z_{\beta}| \geq |z_{\beta}|$, with the possibility of $>$ due to the phenomenon illustrated in \exref{funny example} (reprised below as \exref{funny example reprised}).

Each $y_{\beta}$ will generate a $\Z/p$--submodule $V_{\beta} \subset V(|y_{\beta}|)$ of dimension at most $p$,  and such modules, together with the $z_{\gamma}$, span $V$.

Define a new filtration $\FF^{\prime \prime}$ of $V$ by letting $V^{\prime \prime}(k)$ be the linear span of all $V_{\beta}$ and $z_{\gamma}$ such that $|y_{\beta}| \leq k$ and $|z_{\gamma}| \leq k$. (This might not be a filtration by sub-$\Z/p$--modules.)  Then, for all $k$,
$$ V^{\prime \prime}(k) \subseteq V(k).$$
It follows that
\begin{equation*}
\begin{split}
e(\FF) \leq e(\FF^{\prime \prime}) &
\leq \sum_{\beta} 2p^{|y_{\beta}|}\dim V_{\beta} + \sum_{\gamma} 2p^{|z_{\gamma}|} - r(V)  \\
  &\leq \sum_{\beta} 2p^{|y_{\beta}|+1} + \sum_{\gamma} 2p^{|z_{\gamma}|} - r(V) \\
    &\leq \sum_{\beta} 2p^{|\bar y_{\beta}|} + \sum_{\gamma} 2p^{|\bar z_{\gamma}|} - r(V_{\Z/p}) \\
    & = e(\FF_{\Z/p}).
\end{split}
\end{equation*}
\end{proof}

\begin{proof}[Proof of \lemref{decomp lemma}]  Filter $V_{\Z/p}$ by letting $F_kV_{\Z/p} = (V(k) + rad(V))/rad(V)$.  Then let $W_{\Z/p} = (soc(V) + rad(V))/rad(V) \subset V_{\Z/p}$ be filtered by letting $F_kW_{\Z/p} = W_{\Z/p} \cap F_kV_{\Z/p}$.  It is easy to choose a filtered complement $U_{\Z/p}$ so that $F_kV_{\Z/p} = F_kW_{\Z/p} \oplus F_kU_{\Z/p}$ as filtered $\Z/p$--vector spaces.

The point is now that, as $rad(V) \subseteq V(0)$, one can choose a lifting
\begin{equation*}
\xymatrix{
& V \ar[d]  \\
W_{\Z/p} \ar[r] \ar@{.>}[ur] & V_{\Z/p} }
\end{equation*}
as filtered vector spaces so that the image is contained in $soc(V)$, and thus can be viewed as a lifting of filtered $\Z/p$--modules.  For if $x + rad(V) = y + rad(V)$ with $x \in V(k)$ and $y \in soc(V)$, then $y \in V(k) \cap soc(V) = soc(V(k))$.  The conclusion of the lemma follows if we let $W$ be the image of such a lifting, and $U$ equal to the filtered $\Z/p$--module generated by any lifting
\begin{equation*}
\xymatrix{
& V \ar[d]  \\
U_{\Z/p} \ar[r] \ar@{.>}[ur] & V_{\Z/p}. }
\end{equation*}
\end{proof}

\begin{ex} \label{funny example reprised} We illustrate how \thmref{filtered inv thm}, and its proof, work when our filtered module $V$ is as in \exref{funny example}.  Thus let $p \geq 3$, and let $\FF$ be the filtration given by having $V(0) = V_2$ diagonally embedded in $V = V(1)=  V_1 \oplus V_3$.  Then $\FF_{\Z/p}$ is the filtration having $V_{\Z/p}(0) = V_1$ embedded as the first factor of $V_{\Z/p} = V_{\Z/p}(1) = V_1 \oplus V_1$.

Corresponding to the elements chosen in the proof of part (b) of the theorem, $V$ has a basis $z,y,x_2,x_1$ satisfying the following.
\begin{itemize}
\item $y$ is a $\Z/p$--module generator of $V_3$.
\item $x_2 = \sigma y - y$, $x_1 = \sigma x_2 - x_2$, and these span $rad(V)$.
\item $V(0) = \langle z, x_1\rangle$, and $\sigma z - z = x_1$.
\item The direct summand $V_1$ is spanned by $z-x_2$.
\item $0 = V_{\Z/p}(0)  \subset V_{\Z/p}(1) = \langle \bar z\rangle \subset \langle \bar z, \bar y\rangle = V_{\Z/p}(2) = V_{\Z/p}$.
\end{itemize}

Part (a) of the theorem then says that the image of
$ \F_p[z,x_1,y^p,x_2^p]^{\Z/p}$ in $\F_p[\bar z, \bar y]$ will be contained in $\F_p[\bar z^p, \bar y^{p^2}]$, and part (b) correctly predicts that
$$ e(\FF) = 4p \leq 2p^2 + 2p-2 = e(\FF_{\Z/p}).$$

The auxiliary filtrations $\FF^{\prime}$ and $\FF^{\prime \prime}$ of $V$ used in the theorem's proof satisfy
$$ \langle z \rangle = V^{\prime \prime}(0) \subset V(0) \subset V^{\prime}(0) = \langle z, x_2, x_1 \rangle.$$
\end{ex}

\end{document}